\documentclass[reqno]{amsart}
\numberwithin{equation}{section}

\usepackage[colorlinks,linkcolor={darkblue}]{hyperref}

\usepackage[nobysame,alphabetic]{amsrefs}
\usepackage{bm}
\usepackage{amssymb}
\usepackage{graphicx} 
\usepackage{amscd}
\usepackage{enumerate}
\usepackage{cancel}
\usepackage[font=footnotesize]{caption}
\usepackage{dsfont}
\usepackage[colorinlistoftodos]{todonotes}
\usepackage{tikz-cd}

\definecolor{darkred}{rgb}{1,0,0} 
\definecolor{darkgreen}{rgb}{0,0.8,0}
\definecolor{darkblue}{rgb}{0,0,0.8}

\hypersetup{colorlinks,
linkcolor=darkblue,
filecolor=darkgreen,
urlcolor=black,
citecolor=darkgreen}

\makeatletter
\providecommand\@dotsep{5}
\makeatother

\definecolor{orange}{RGB}{253,85,0}
\definecolor{darkgreen}{RGB}{0,95,10}

 \newcommand{\ev}{\mathrm{ev}}
 
 \newcommand{\alphalim}{\alpha\mbox{-}\!\lim}
 \newcommand{\omegalim}{\omega\mbox{-}\!\lim}
 \newcommand{\N}{\mathds{N}}
 \newcommand{\Z}{\mathds{Z}}
 \newcommand{\R}{\mathds{R}}
 \newcommand{\C}{\mathds{C}}

 \newcommand{\A}{\mathcal{A}}

 \newcommand{\GG}{\mathcal{G}}
 \newcommand{\XX}{\mathcal{X}}
 \newcommand{\LL}{\mathcal{L}}
 \newcommand{\MM}{\mathcal{M}}
 \newcommand{\JJ}{\mathcal{J}}

 \newcommand{\CC}{\mathcal{C}}

 \newcommand{\inj}{\mathrm{inj}}
 
 \DeclareMathOperator{\supp}{supp}

 \DeclareMathOperator{\rank}{rank}

 \DeclareMathOperator*{\ttoup}{\llongrightarrow}

\DeclareRobustCommand{\llongrightarrow}{\relbar\joinrel\relbar\joinrel\rightarrow}

 \theoremstyle{plain}
 \newtheorem{MainThm}{Theorem}
 \newtheorem{MainCor}[MainThm]{Corollary}
 
 \newtheorem{Thm}{Theorem}[section]
 \newtheorem{Prop}[Thm]{Proposition}
 \newtheorem{Lemma}[Thm]{Lemma}
 \newtheorem{Cor}[Thm]{Corollary}
 
 \theoremstyle{definition}
 
 \newtheorem{Remark}[Thm]{Remark}

\title{Closed geodesics and the first Betti number} 

\author{Gonzalo Contreras}
\address{Gonzalo Contreras\newline\indent 
Centro de Investigaci\'on en Matem\'aticas\newline\indent Jalisco S/N, Valenciana, 36023 Guanajuato, GTO, Mexico}
\email{gonzalo@cimat.mx}

\author{Marco Mazzucchelli}
\address{Marco Mazzucchelli\newline\indent CNRS, UMPA, \'Ecole Normale Sup\'erieure de Lyon\newline\indent 46 all\'ee d'Italie, 69364 Lyon, France}
\email{marco.mazzucchelli@ens-lyon.fr}

\date{July 3, 2024}



\thanks{Gonzalo Contreras is partially supported by CONACYT, Mexico, grant A1-S-10145. Marco Mazzucchelli is partially supported by the ANR grants CoSyDy (ANR-CE40-0014) and COSY (ANR-21-CE40-0002).}

\begin{document}

\begin{abstract}
We prove that, on any closed manifold of dimension at least two with non-zero first Betti number, a $C^\infty$ generic Riemannian metric has infinitely many closed geodesics, and indeed closed geodesics of arbitrarily large length. We derive this existence result combining a theorem of Mañé together with the following new theorem of independent interest: the existence of minimal closed geodesics, in the sense of Aubry--Mather theory, implies the existence of a transverse homoclinic, and thus of a horseshoe, for the geodesic flow of a suitable $C^\infty$-close Riemannian metric.
\end{abstract}

\maketitle

\section{Introduction}\label{s:intro}

\subsection{Background}\label{ss:background}
A long standing conjecture in Riemannian geometry asserts that any closed Riemannian manifold of dimension at least two has infinitely many closed geodesics. This conjecture holds for any simply connected closed Riemannian manifold whose rational cohomology ring is not generated by a single element, thanks to a combination of results of Gromoll and Meyer \cite{Gromoll:1969aa} and Vigué-Poirrier and Sullivan \cite{Sullivan:1975aa,Vigue-Poirrier:1976aa}. For non-simply connected Riemannian manifolds, the conjecture was confirmed by Bangert and Hingston \cite{Bangert:1984aa} for closed manifolds whose fundamental group is infinite abelian (the most difficult case being $\Z$), and later generalized by Taimanov \cite{Taimanov:1985aa} to larger classes of closed manifolds, including those with infinite solvable fundamental group. The conjecture also holds for any Riemannian surface, and most notably for any Riemannian 2-sphere, thanks to a combination of results of Bangert \cite{Bangert:1993aa} and Franks \cite{Franks:1992aa} or, alternatively, Hingston \cite{Hingston:1993aa}. To the best of the authors' knowledge, these are the last results confirming the conjecture for any Riemannian metric on certain classes of manifolds. Among the remaining cases, the conjecture is still open for closed manifolds of dimension at least three having the rational cohomology of a compact rank-one symmetric space $S^n$, $\C P^n$, $\mathds H P^n$, or $\mathrm{Ca}P^2$.

A result of Hingston \cite{Hingston:1984aa}, later reproved by Rademacher \cite{Rademacher:1989aa} with a different argument, asserts that a $C^4$-generic Riemannian metric on any simply connected closed  manifold with the rational cohomology of a compact rank-one symmetric space has infinitely many closed geodesics. When the fundamental group is infinite and non-abelian, $C^4$-generic existence results were proved only for specific classes of closed  manifolds, see \cite{Ballmann:1981aa,Rademacher:2022aa,Shelukhin:2023aa} and references therein. The general case of closed manifolds with infinite non-abelian fundamental group is still open.

\subsection{Main results}
In this paper, we prove a new existence result for closed geodesics and for homoclinics to closed geodesics, by means of Aubry--Mather theory \cite{Bangert:1990aa,Contreras:1999aa,Fathi:Weak_KAM_theorem_in_Lagrangian_dynamics,Sorrentino:2015aa}. We provide the statements after recalling some relevant definitions. 

We consider a closed Riemannian manifold $(M,g)$ of dimension at least two with non-zero first Betti number. This latter condition is equivalent to the non-vanishing of the first de Rham cohomology group $H^1(M;\R)$. For each closed 1-form $\sigma$ on $M$, we can associate to each $W^{1,2}$ curve $\gamma:[0,\tau]\to M$ an action
\begin{align}
\label{e:action}
 \A_{\sigma}(\gamma)=\frac1\tau \int_0^\tau \Big(\tfrac12\|\dot\gamma(t)\|_g^2 - \sigma(\dot\gamma(t)) \Big)\,dt.
\end{align}
When $\gamma$ is a loop, meaning that $\gamma(0)=\gamma(\tau)$, the value $A_{\sigma}(\gamma)$ does not depend on the specific choice of $\sigma$, but only on the cohomology class $[\sigma]\in H^1(M;\R)$. From now on, in order to simplify the notation, we will omit the brackets and write  $\sigma$ for the cohomology class as well.

Throughout this paper, by a geodesic $\gamma:\R\to M$ we always mean a non-constant solution of $\nabla_t\dot\gamma\equiv0$, where $\nabla_t$ is the Levi-Civita covariant derivative of $(M,g)$.
A closed geodesic is a geodesic $\gamma:\R\to M$ such that $\gamma=\gamma(\tau_\gamma+\cdot)$ for some minimal period $\tau_\gamma>0$. We associate to any such closed geodesic the Riemannian length $\LL(\gamma):=\tau_\gamma\|\dot\gamma\|_g$ and the action $\A_{\sigma}(\gamma|_{[0,\tau_\gamma]})$. A closed geodesic $\gamma$ is called \emph{minimal} (in the sense of Aubry--Mather theory \cite{Bangert:1990aa}) when, for some non-zero $\sigma\in H^1(M;\R)$, we have
\begin{align*}
\A_{\sigma}(\gamma|_{[0,\tau_\gamma]})
=
\inf_{\zeta}\A_{\sigma}(\zeta),
\end{align*}
where the infimum ranges over all $\tau>0$ and $W^{1,2}$ loops $\zeta:[0,\tau]\to M$, $\zeta(0)=\zeta(\tau)$. We say that $\gamma$ is $\sigma$-minimal or $(g,\sigma)$-minimal if we need to specify the co\-homology class and the Riemannian metric.

Any closed geodesic $\gamma$ lifts to a periodic orbit $\dot\gamma$ of the geodesic flow on the sphere tangent bundle of radius $\|\dot\gamma\|_g$. When $\gamma$ is hyperbolic, meaning that $\dot\gamma$ is a hyperbolic periodic orbit of the geodesic flow, it may admit transverse homoclinics, that is, geodesics distinct from $\gamma$ and whose lifts to the sphere tangent bundle lie on transverse intersection points of the stable and unstable manifolds of $\dot\gamma$. By a classical result from hyperbolic dynamics \cite[Theorem~6.5.2]{Fisher:0aa}, the presence of a transverse homoclinic implies the existence of a horseshoe for the geodesic flow. This further implies that the geodesic flow has positive topological entropy and exponential growth of the periodic orbits, and in particular that there are infinitely many closed geodesics of arbitrarily large length. 

The following is the main result of this article.
\vspace{2pt}

\begin{MainThm}\label{mt:main}
Let $(M,g_0)$ be a closed Riemannian manifold of dimension at least two, with a minimal closed geodesic $\gamma$. Then there exists a Riemannian metric $g$ arbitrarily $C^\infty$-close to $g_0$ such that $\gamma$ is a hyperbolic closed geodesic of $g$ with a transverse homoclinic.
\end{MainThm}\vspace{1pt}

The existence of a transverse homoclinic to a hyperbolic closed geodesic after a $C^2$-small perturbation of the Riemannian metric was established by the first author in \cite{Contreras:2010aa} for those closed manifolds of dimension at least two on which a $C^2$-generic Riemannian metric has infinitely many closed geodesics\footnote{In the main theorems in \cite{Contreras:2010aa}, the requirement that a $C^2$-generic Riemannian metric on the considered closed manifold must have infinitely many closed geodesics is missing due to an omission.}. In particular, the theorem holds for simply connected closed manifolds (which do not admit minimal closed geodesics). 
Our Theorem~\ref{mt:main}, instead, employs in an essential way a minimal closed geodesic, and achieves the transverse homoclinic with a perturbation of the Riemannian metric in the finer $C^\infty$ topology.

In the general setting of Tonelli Hamiltonians, but under the stronger assumption that the first Betti number of the underlying closed manifold $M$ is at least two, the analogue of Theorem~\ref{mt:main} was established by the first author and  Paternain \cite[Cor.~2]{Contreras:2002aa} (in the Tonelli setting, the Hamiltonian function is perturbed with a potential). In a similar spirit, results on homoclinics were also obtained by Bolotin \cite{Bolotin:1995aa,Bolotin:1995ab,Bolotin:1997aa} using different methods. The essential novelty of our Theorem~\ref{mt:main} is that it allows $M$ to have first Betti number equal to one. In particular, the hardest case is when the fundamental group $\pi_1(M)$ is isomorphic to $\Z$, for which the quest for homoclinics requires a min-max scheme inspired by the above mentioned result of Bangert and Hingston \cite{Bangert:1984aa}.

If there are no $\sigma$-minimal closed geodesics for some non-zero cohomology class $\sigma$, a result of Mañé \cite[Th.~F]{Mane:1996aa} asserts that, even without perturbing the Riemannian metric, there exist infinitely many closed geodesics of arbitrarily large length. This, combined with Theorem~\ref{mt:main}, implies the following corollary.
We denote by $\GG^k(M)$ the space of smooth Riemannian metrics on $M$, endowed with the $C^k$ topology.\vspace{2pt}

\begin{MainCor}\label{mc:multiplicity}
Let $M$ be a closed manifold of dimension at least two with non-zero first Betti number. Then, for each $2\leq k\leq\infty$, there exists an open and dense subset of $\GG^k(M)$ such that every Riemannian metric therein admits infinitely many closed geodesics of arbitrarily large length.
\end{MainCor}\vspace{1pt}

We will actually prove a slight generalization of Theorem~\ref{mt:main}, allowing the assumptions to be satisfied only by a finite cover of the closed manifold $M$ (Theorem~\ref{t:main_cover}), and derive a stronger version of the latter corollary (Corollary~\ref{c:multiplicity}).

\subsection{Organization of the paper} 
In Section~\ref{s:AM}, after recalling the needed background from Aubry--Mather theory, we prove Theorem~\ref{mt:main} and Corollary~\ref{mc:multiplicity}. In the Appendix, we prove a perturbation result for closed geodesics that is needed in the proof of Theorem~\ref{mt:main}.

\subsection{Acknowledgments}
We are grateful to the anonymous referees for carefully reading the manuscript and providing helpful suggestions and corrections.

\section{Aubry--Mather theory}
\label{s:AM}

\subsection{Preliminaries}
The proof of Theorem~\ref{mt:main} requires some tools from Aubry--Mather theory \cite{Bangert:1990aa,Contreras:1999aa,Fathi:Weak_KAM_theorem_in_Lagrangian_dynamics,Sorrentino:2015aa}. Let $(M,g)$ be a closed Riemannian manifold of dimension at least two. 
We consider the geodesic flow \[\phi^t=\phi_g^t:TM\to TM\] defined on the whole tangent bundle. Its orbits have the form
$\phi^t(\dot\gamma(0))=\dot\gamma(t)$,
where $\gamma:\R\to M$ is a geodesic or a constant curve. 
We denote by $\MM$ the space of Borel probability measures $\mu$ on $TM$ that are closed, meaning that
\begin{align*}
 \int_{TM} df\,d\mu=0,\qquad\forall f\in C^1(M).
\end{align*}
Within $\MM$, we have two important classes of measures:
\begin{itemize}
 \item All those probability measures $\mu$ on $TM$ that are invariant under the geodesic flow, i.e.~$\phi^t_*\mu=\mu$ for all $t\in\R$.
 
 \item All those Borel probability measures $\mu_\gamma$ uniformly distributed along a continuous and piecewise smooth loop $\gamma:[0,\tau]\to M$, $\gamma(0)=\gamma(\tau)$, i.e.
 \begin{align*}
 \int_{TM} F\,d\mu_\gamma := \frac1\tau\int_0^\tau F(\dot\gamma(t))\,dt,
 \qquad\forall F\in C^0(TM).
\end{align*}
\end{itemize}
Any $\mu\in\MM$ has a rotation vector $\rho(\mu)\in H_1(M;\R)$, which is defined via the duality with de Rham cohomology classes $\sigma\in H^1(M;\R)$ by
\begin{align*}
\langle\sigma,\rho(\mu)\rangle = \int_{TM} \sigma(v)\,d\mu(v).
\end{align*}
Here, as well as later on, within the integral we employ an arbitrary closed 1-form representing $\sigma$, which we still denote by $\sigma$ with a slight abuse of notation. Since $\mu$ is a closed measure, the value of the integral is independent of the choice of such a closed 1-form.

For each $\sigma\in H^1(M;\R)$, we consider the Lagrangian action functional
\begin{align*}
\A_{\sigma}=\A_{g,\sigma}:\MM\to(-\infty,\infty],
\qquad
 \A_{\sigma}(\mu) = \int_{TM} \Big(\tfrac12\|v\|_g^2 - \sigma(v) \Big)\,d\mu(v).
\end{align*}
Notice that
$\A_\sigma(\mu)=\A_0(\mu)-\langle\sigma,\rho(\mu)\rangle$. The notation for the action $\A_\sigma$ is consistent with the one introduced in~\eqref{e:action}: for each continuous and piecewise smooth loop $\gamma:[0,\tau]\to M$, $\gamma(0)=\gamma(\tau)$,  with associated probability measure $\mu_\gamma$, we have 
\[\A_{\sigma}(\gamma)=\A_{\sigma}(\mu_\gamma).\]

The action functional $\A_{\sigma}$ is bounded from below and achieves its minimum on $\MM$. Any minimizer turns out to be invariant under the geodesic flow, and is called a $\sigma$-minimal measure (or a $(g,\sigma)$-minimal measure if we need to specify the Riemannian metric). 
The space of $\sigma$-minimal measures, endowed with the weak-$*$ topology, is compact and convex, and its extremal points are ergodic measures.
Mather alpha function $\alpha=\alpha_g:H^1(M;\R)\to\R$ is defined by
\begin{align*}
\alpha(\sigma):=-\min_{\MM}\A_\sigma.
\end{align*}
Alternatively, instead of minimizing over the space of closed measures, Mather alpha function $\alpha:H^1(M;\R)\to\R$ is also characterized by
\begin{align}\label{e:alpha}
\alpha(\sigma)=-\inf_{\gamma}\A_{\sigma}(\gamma),
\end{align}
where the infimum ranges over all $\tau\geq0$ and $W^{1,2}$ loops $\gamma:[0,\tau]\to M$, $\gamma(0)=\gamma(\tau)$. 
If $H^1(M;\R)$ is non-trivial, the function $\alpha$ is non-negative, convex, superlinear, and satisfies $\alpha(0)=0$. These properties hold more generally for the alpha function associated to any Tonelli Lagrangian. In the specific case of geodesic flows, we also have the following.

\begin{Lemma}\label{l:min_alpha}
The origin is a strict local minimum of Mather alpha function, i.e. $\alpha(\sigma)>0$ for all non-zero $\sigma\in H^1(M;\R)$.
\end{Lemma}

\begin{proof}
Consider a non-zero cohomology class $\sigma\in H^1(M;\R)$, and fix any smooth loop $\gamma:[0,\tau]\to M$, $\gamma(0)=\gamma(\tau)$ such that
\begin{align*}
 \int_\gamma \sigma>0.
\end{align*}
Since this latter integral and the length 
\[
\LL(\gamma)
:=
\int_0^\tau \|\dot\gamma(t)\|_g\,dt
\] 
are independent of the parametrization of $\gamma$, we can assume that the speed $\|\dot\gamma\|_g$ is constant and sufficiently small so that
\[
\frac12 \LL(\gamma)\|\dot\gamma\|_g< \int_\gamma \sigma .
\]
This implies
\begin{align*}
\A_\sigma(\gamma)
=
\frac{1}{\tau} \left(
\frac12 \LL(\gamma)\|\dot\gamma\|_g 
- 
\int_\gamma \sigma \right)
<
0,
\end{align*}
and therefore $\alpha([\sigma])>0$ according to~\eqref{e:alpha}.
\end{proof}

Notice that, for the zero cohomology class $\sigma=0$, the infimum in~\eqref{e:alpha} is always a minimum, and it is achieved only by the constant curves. Instead, for each non-zero $\sigma\in H^1(M;\R)$, a $W^{1,2}_{\mathrm{loc}}$ periodic curve $\gamma:\R\to M$ of minimal period $\tau_\gamma$ is a $\sigma$-minimal closed geodesic if and only if $\gamma|_{[0,\tau_\gamma]}$ achieves the minimum in~\eqref{e:alpha}. In this case, in particular $\gamma$ is smooth, and the measure $\mu_{\gamma}$ associated with $\gamma|_{[0,\tau_\gamma]}$ is $\sigma$-minimal.

Let $\pi:TM\to M$ be the base projection of the tangent bundle. Mather's graph theorem \cite[Theorem~2]{Mather:1991aa} asserts that, for any $\sigma$-minimizing measure $\mu$, the restriction $\pi|_{\supp(\mu)}$ is an injective bi-Lipschitz map onto its image. Moreover, by a theorem due to Carneiro \cite{Carneiro:1995aa}, $\supp(\mu)$ is contained in the sphere tangent bundle 
\[S^r\!M=\big\{v\in TM\ \big|\ \|v\|_g=r\big\},\]
for $r^2/2=\alpha(\sigma)$. In particular, any $\sigma$-minimal closed geodesic $\gamma$   has speed $\|\dot\gamma\|_g\equiv r$ and is simple, i.e.~the restriction $\gamma|_{[0,\tau_\gamma)}$ is an injective map, where $\tau_\gamma>0$ is the minimal period of $\gamma$.

The Riemannian metric $g$ and a closed 1-form $\sigma$ on $M$ define a Tonelli Lagrangian $L:TM\to\R$ and a dual Tonelli Hamiltonian $H:T^*M\to\R$ by
\begin{align*}
 L(v)=\tfrac12 \|v\|_{g}^2 - \sigma(v),
 \qquad
 H(p)=\tfrac12\|p+\sigma\|_g^2.
\end{align*}
These functions are related by the Fenchel inequality $H(p)+L(v)\geq p(v)$.
A theorem due to Fathi and Siconolfi \cite{Fathi:2004aa}, asserts that there exists a $C^{1}$ function $u:M\to\R$ satisfying the Hamilton-Jacobi inequality $H\circ du\leq\alpha(\sigma)$.

\subsection{Proofs of the theorems}

Before carrying out the proof of Theorem~\ref{mt:main}, for the reader's convenience we first provide the short proof of a theorem due to Mañé \cite[Th.~F]{Mane:1996aa} in the special case of geodesic flows, which we will need to derive Corollary~\ref{mc:multiplicity}. 
We say that a closed geodesic $\gamma$ of $(M,g)$ has non-zero real homology when $[\gamma|_{[0,\tau_\gamma]}]\neq0$ in $H_1(M;\R)$.

\begin{Thm}\label{t:aperiodic}
Let $(M,g)$ be a closed Riemannian manifold. If there exists a non-zero $\sigma\in H^1(M;\R)$ not admitting any $\sigma$-minimal closed geodesic, then there exist infinitely many closed geodesics of arbitrarily large length and non-zero real homology.
\end{Thm}

\begin{proof}
Let $\mu$ be an ergodic $\sigma$-minimal measure. By Poincar\'e recurrence theorem and Birkhoff ergodic theorem, there exists $v\in\supp(\mu)$ that is recurrent for the geodesic flow $\phi^t$ and regular for the Birkhoff average. Namely, there exists a sequence of positive real numbers $\tau_n\to\infty$ such that $\phi^{\tau_n}(v)\to v$, and 
\begin{align*}
 \lim_{n\to\infty} \frac1{\tau_n}\int_0^{\tau_n} F(\phi^t(v))\,dt
 =
 \int_{TM} F\,d\mu,\qquad\forall F\in L^1(TM,\mu).
\end{align*}
We fix a quantity $\delta>0$ such that $\delta\|v\|_g$ is smaller than the injectivity radius $\inj(M,g)$, and consider the geodesic arc $\eta_n:[0,\tau_n-\delta]\to M$, $\eta_n(t)=\pi(\phi^t(v))$. For all $n$ large enough, there exists a unique geodesic arc $\zeta_n:[0,\delta]\to M$ of length smaller than $\inj(M,g)$ joining $\eta_n(\tau_n-\delta)$ and $\eta_n(0)=\pi(v)$. 
Notice that the action $\A_\sigma(\zeta_n)$ is uniformly bounded from above for all $n$.
The concatenation $\eta_n*\zeta_n:[0,\tau_n]\to M$ is a loop with action
\begin{align}
\label{e:action_eta*zeta}
 \A_\sigma(\eta_n*\zeta_n)
 =
 \frac{(\tau_n-\delta)\A_\sigma(\eta_n) +\delta\A_\sigma(\zeta_n)}{\tau_n}
 \ttoup_{n\to\infty}
 \A_\sigma(\mu)=-\alpha(\sigma).
\end{align}
Let $\Omega_n$ be the space of $W^{1,2}$ loops $\zeta:[0,\tau_n]\to M$, $\zeta(0)=\zeta(\tau_n)$, and  $\gamma_n\in\Omega_n$ a loop that  minimizes $\A_\sigma|_{\Omega_n}$, i.e.
\begin{align*}
 \A_\sigma(\gamma_n)\leq\A_\sigma(\zeta),
 \qquad\forall \zeta\in\Omega_n. 
\end{align*}
In particular, 
\begin{align}
\label{e:action_gamma_eta}
 \A_\sigma(\gamma_n)\leq\A_\sigma(\eta_n*\zeta_n).
\end{align}
Each $\gamma_n$ is either a constant curve (with action $\A_\sigma(\gamma_n)=0$) or a closed geodesic (namely a geodesic loop such that $\dot\gamma_n(0)=\dot\gamma_n(\tau_n)\neq0$).
Up to extracting a subsequence, the probability measure $\mu_{\gamma_n}$ converges in the weak-$*$ topology to an invariant probability measure $\nu$, and so do the corresponding actions $\A_\sigma(\gamma_n)\to\A_\sigma(\nu)$. By~\eqref{e:action_eta*zeta} and~\eqref{e:action_gamma_eta}, we infer $\A_\sigma(\nu)\leq\A_\sigma(\mu)$, and therefore $\A_\sigma(\nu)=\A_\sigma(\mu)$. Namely, $\nu$ is a $\sigma$-minimal measure. Since $\sigma\neq0$, Lemma~\ref{l:min_alpha} implies that $\alpha(\sigma)>0$. By
\begin{align*}
 \langle \sigma,\rho(\nu) \rangle
 =
 \A_0(\nu) - \A_\sigma(\nu)
 =
 \A_0(\nu) + \alpha(\sigma)
 \geq
 \alpha(\sigma)
 >
 0,
\end{align*}
we infer that $\rho(\nu)\neq0$. 
Since by assumption there are no $\sigma$-minimal closed geodesics, we have the strict inequality $\A_\sigma(\gamma_n)>\A_\sigma(\nu)$. This, together with the convergence $\A_\sigma(\gamma_n)\to\A_\sigma(\nu)$, implies that the family $\gamma_n$, for $n\geq0$, contains infinitely many closed geodesics. Since $[\gamma_n]\to\rho(\nu)\neq0$, the closed geodesics $\gamma_n$ have non-zero real homology for all $n$ large enough. Since the support of $\nu$ is contained in the sphere tangent bundle $S^r\!M$ for $r^2/2=\alpha(\sigma)$, the weak-$*$ convergence $\mu_{\gamma_n}\to\nu$ implies that $\|\dot\gamma_n\|_g\to r$. Let $\tau_{\gamma_n}\leq\tau_n$ be the minimal period of the closed geodesic $\gamma_n$, which is the minimal positive number such that $\gamma_n(0)=\gamma_n(\tau_{\gamma_n})$ and $\dot\gamma_n(0)=\dot\gamma_n(\tau_{\gamma_n})$. The sequence $\tau_{\gamma_n}$ must diverge, for otherwise $\gamma_n$ would converge to a $\sigma$-minimal closed geodesic. Therefore the lengths $\LL(\gamma_n)=\tau_{\gamma_n}\|\dot\gamma_n\|_g$ diverge.
\end{proof}

We will infer our main Theorem~\ref{mt:main} from the following statement, which under the same assumptions provide a (not necessarily transverse) homoclinic after an explicit conformal perturbation of the Riemannian metric. The transversality of the homoclinic will then be achieved by invoking a perturbation result of Petroll \cite{Burns:2002aa}.

\begin{Thm}\label{t:homoclinic}
Let $(M,g_0)$ be a closed Riemannian manifold of dimension at least two, with a minimal closed geodesic $\gamma$.
Let $\rho:M\to[0,\infty)$ be any smooth function such that $\rho(x)=0$ and $d^2\rho(x)[v,v]>0$ for all $x\in\gamma$ and $v\in T_xM\setminus\{0\}$ orthogonal to $\gamma$, and $\rho(y)>0$ for all $y\in M\setminus\gamma$. Then $\gamma$ is a hyperbolic closed geodesic of $e^\rho g_0$ with a homoclinic.
\end{Thm}

\begin{proof}
Let $\gamma$ be a $(g_0,\sigma)$-minimal closed geodesic, and as usual we denote by $\tau_\gamma>0$ its minimal period. In particular, $\gamma$ is a simple closed geodesic without conjugate points. We set $x_0:=\gamma(0)=\gamma(\tau_\gamma)$. For each integer $n\geq1$, we see the loop $\gamma|_{[0,n\tau_\gamma]}$ as a representative of an element of the fundamental group $\pi_1(M,x_0)$.  

By the already mentioned theorem of Fathi and Siconolfi \cite{Fathi:2004aa}, there exists a $C^{1}$ function $u:M\to\R$ that satisfies the Hamilton-Jacobi inequality 
$H_0\circ du\leq\alpha_{g_0}(\sigma)$, 
where $H_0:T^*M\to\R$ is the Tonelli Hamiltonian dual to the Tonelli Lagrangian
\[
L_0:TM\to\R,
\qquad 
L_0(v)=\tfrac12 \|v\|_{g_0}^2 - \sigma(v).
\]
This, together with the Fenchel inequality $H_0(p)+L_0(v)\geq p(v)$, implies
\begin{align*}
F_0(v):=L_0(v) - du(\pi(v))v + \alpha_{g_0}(\sigma)\geq0,\qquad\forall v\in T_xM,
\end{align*}
where $\pi:TM\to M$ is the base projection. Since
\begin{align*}
 \int_{0}^{\tau_\gamma} F_0(\dot\gamma(t))\,dt
 =
 \tau_\gamma\big( \A_{g_0,\sigma}(\gamma) + \alpha_{g_0}(\sigma) \big) = 0,
\end{align*}
we have $F_0\circ\dot\gamma\equiv0$.

Let $\rho:M\to[0,\infty)$ be any smooth function such that $\rho(x)=0$ and $d^2\rho(x)[v,v]>0$ for all $x\in\gamma$ and $v\in T_xM\setminus\{0\}$ orthogonal to $\gamma$, and $\rho(y)>0$ for all $y\in M\setminus\gamma$.
By Proposition~\ref{p:make_hyperbolic}, $\gamma$ is a hyperbolic closed geodesic for the Riemannian metric \[g:=e^\rho g_0.\] 
Since $\A_{g,\sigma}\geq \A_{g_0,\sigma}$ and $\A_{g,\sigma}(\gamma)= \A_{g_0,\sigma}(\gamma)$, we infer that $\gamma$ is a $(g,\sigma)$-minimal closed geodesic, and therefore
\begin{align*}
 \alpha:=\alpha_{g_0}(\sigma)=\alpha_{g}(\sigma)=-\A_{g,\sigma}(\gamma)>0.
\end{align*}
We introduce the non-negative continuous Tonelli Lagrangian
\begin{align*}
 F:TM\to[0,\infty),\qquad F(v):=\tfrac12 \|v\|_{g}^2 - \sigma(v) - du(x)v + \alpha.
\end{align*}
Notice that $F$ is identically equal to $\alpha>0$ along the zero section, and $F(v)>0$ for all $v\in TM$ such that $\pi(v)\not\in\gamma$. Moreover, for each $t\in[0,\tau_\gamma]$, we have $F(r\dot\gamma(t))=0$ if and only if $r=1$.

For each compact interval $[\tau_1,\tau_2]\subset\R$ and for each $W^{1,2}$ curve $\zeta:[\tau_1,\tau_2]\to M$, we set
\begin{align*}
 a(\zeta)
 :=
 \int_{\tau_1}^{\tau_2} F(\dot\zeta(t))\,dt\geq0.
 \end{align*}
Notice that $a(\gamma|_{[0,\tau_\gamma]})=0$. The following lemma is crucial, and requires two distinct proofs for the cases $\pi_1(M,x_0)\not\cong\Z$ and $\pi_1(M,x_0)\cong\Z$. From now on, all the geodesics are associated with the Riemannian metric $g$, unless we specify otherwise. Two geodesics are said to be geometrically distinct when their images into the Riemannian manifold are distinct.

\begin{Lemma}\label{l:geodesic_with_finite_action}
There exists a geodesic $\zeta:\R\to M$ geometrically distinct from $\gamma$ such that 
\begin{align}
\label{e:bound_integral}
 \int_{-\infty}^\infty F(\dot\zeta(t))\,dt < \infty.
\end{align}
\end{Lemma}

Postponing the proof of this lemma, let us first complete the proof of Theorem~\ref{t:homoclinic}. We shall show that $\zeta$ is a homoclinic to the closed geodesic $\gamma$.
For each $\epsilon>0$, we denote by $N_\epsilon\subset M$ the open tubular neighborhood of $\gamma$ of radius $\epsilon>0$, measured with respect to the Riemannian metric $g$. Since $F$ is strictly positive outside $TN_\epsilon$ and coercive, in particular we have
\begin{align*}
\delta_\epsilon:=\min_{T(M\setminus N_{\epsilon})} F > 0.
\end{align*}
Assume that, on some interval $[t_1,t_2]\subset\R$, the geodesic arc $\zeta|_{[t_1,t_2]}$ crosses the shell $N_{2\epsilon}\setminus N_{\epsilon}$, so that it has length 
\[\|\dot\zeta\|_{g}(t_2-t_1)\geq\epsilon\] and action
\begin{align}
\label{e:crossing_action}
 a(\zeta|_{[t_1,t_2]}) \geq (t_2-t_1)\delta_\epsilon\geq\frac{\epsilon\, \delta_\epsilon}{\|\dot\zeta\|_{g}}=:\rho_\epsilon.
\end{align}
Since $F$ is continuous and non-negative, there exists $s_0>0$ such that 
\begin{align}
\label{e:action_almost_full}
a(\zeta|_{[-s_0,s_0]})
>
\int_{-\infty}^\infty F(\dot\zeta(t))\,dt
-\rho_\epsilon, 
\end{align}
and $s_1>s_0$ such that $F(\zeta(s_1))<\delta_\epsilon$. Therefore $\zeta(s_1)\in N_{\epsilon}$. The  inequalities~\eqref{e:crossing_action} and \eqref{e:action_almost_full} imply that $\zeta(t)\in N_{2\epsilon}$ for all $t>s_1$. Analogously, $\zeta(-t)\in N_{2\epsilon}$ for all $t>0$ large enough. Overall, by sending $\epsilon\to0$, this argument shows that the distance of $\zeta(t)$ to the closed geodesic $\gamma$ tends to 0 as $|t|\to\infty$. Therefore, $\dot\zeta$ must have the $\alpha$-limit and $\omega$-limit
\[
\alphalim\dot\zeta=r_\alpha\dot\gamma,
\qquad
\omegalim\dot\zeta=r_\omega\dot\gamma,
\]
where $|r_\alpha|=|r_\omega|=\|\dot\zeta\|_{g}/\|\dot\gamma\|_{g}$. Since $F(r\dot\gamma(t))>0$ for all $t\in[0,\tau_\gamma]$ and $r\neq1$, the finiteness of the integral~\eqref{e:bound_integral} implies $r_\alpha=r_\omega=1$. Therefore 
\[\alphalim\dot\zeta=\omegalim\dot\zeta=\dot\gamma,\] 
that is, $\zeta$ is a homoclinic to $\gamma$.
\end{proof}

\begin{proof}[Proof of Lemma~\ref{l:geodesic_with_finite_action} in the case $\pi_1(M,x_0)\not\cong\Z$]
Let $N$ be an open tubular neighborhood of the simple closed geodesic $\gamma$. We denote the inclusion by $i:N\hookrightarrow M$. Since $N$ is homotopy equivalent to a circle, it has fundamental group $\pi_1(N,x_0)\cong\Z$. Since $\pi_1(M,x_0)\not\cong\Z$ and $H^1(M;\R)\neq0$, the homomorphism $i_*:\pi_1(N,x_0)\to \pi_1(M,x_0)$ is not surjective (indeed this latter condition would be enough to carry out the remainder of the proof). We set 
 $G:=i_*(\pi_1(N,x_0))$,
and fix a homotopy class $h\in \pi_1(M,x_0)\setminus G$.
For each $T>0$, consider the loop space
\begin{align*}
 \Omega_T
 :=
 \Big\{
 \zeta:[0,\tau]\ttoup^{W^{1,2}} M \ \Big|\ 
 \zeta(0)=\zeta(\tau)=x_0,\ 0<\tau\leq T,\ [\zeta]\in GhG
 \Big\}.
\end{align*}
Namely, $\Omega_T$ consists of those loops based at $x_0$, defined on an interval of length at most $T$, and representing a non-trivial element of the fundamental group $\pi_1(M,x_0)$ of the form $[\gamma]^jh[\gamma]^k$ for some $j,k\in\Z$.

The functional $a|_{\Omega_{T}}$ achieves its minimum at some geodesic loop $\zeta_T:[0,t_T]\to M$, with $0<t_T\leq T$, which is not necessarily unique. We choose one such minimizer with the highest possible period $t_T$, so that the function $T\mapsto t_T$ is non-decreasing.
We fix a constant $c>0$ large enough so that $F(v)\geq\tfrac14\|v\|_{g}^2-c$ for all $v\in TM$,
and therefore 
\[
a(\zeta_T)\geq t_T\big(\tfrac14\|\dot\zeta_T\|_{g}^2-c\big).
\]
Since the function $T\mapsto a(\zeta_T)$ is non-increasing, we have
\[
a_\infty
 :=
 \lim_{T\to\infty} a(\zeta_T)<\infty,
\] 
and for all $T\geq1$  we have 
\begin{align}
\label{e:speed_bound}
 \frac14\|\dot\zeta_T\|_{g}^2
 \leq
 \frac{a(\zeta_T)}{t_T}+c\leq \frac{a(\zeta_1)}{t_1}+c.
\end{align}
Since $[\zeta_T]\in G h G$ and $h\not\in G$, we have that  $[\zeta_T]\not\in G$. Therefore, there exists $s_T$ such that $\zeta_T(s_T)\not\in N$. The uniform bound~\eqref{e:speed_bound} allows us to extract a diverging sequence $T_n\to\infty$ such that, if we set $\zeta_n:=\zeta_{T_n}$, $s_n:=s_{T_n}$, and $t_n:=t_{T_n}$, we have
\[
x_n:=\zeta_{n}(s_{n})\to x,
\qquad
v_n:=\dot \zeta_{n}(s_{n})\to v.
\]
Let $\zeta:\R\to M$ be the geodesic such that $\zeta(0)=x$ and $\dot\zeta(0)=v$.

We claim that
\begin{align*}
 \lim_{n\to\infty}\min\{s_n,t_n-s_n\}\to\infty.
\end{align*}
Assume by contradiction that $s_n$ is uniformly bounded from above. Up to extracting a subsequence, we have $s_n\to s>0$.  Since $a(\gamma|_{[0,\tau_\gamma]})=0$, we have 
\begin{align*}
 a(\gamma|_{[0,\tau_\gamma]}*\zeta_n|_{[0,s_n]})=a(\zeta_n|_{[0,s_n]}),
\end{align*}
where $*$ denotes the concatenation of paths. Notice that $\gamma|_{[0,\tau_\gamma]}*\zeta_n|_{[0,s_n]}$ is not a geodesic, since it has a corner at $\gamma(\tau_\gamma)=\zeta_n(0)$. For each $\epsilon>0$, we introduce the space
\begin{align*}
 \Upsilon_{n,\epsilon}:=\Big\{ \lambda:[-\epsilon,\epsilon]\ttoup^{W^{1,2}}M\ \Big|\ \lambda(-\epsilon)=\gamma(\tau_\gamma-\epsilon),\ \lambda(\epsilon)=\zeta_n(\epsilon) \Big\}.
\end{align*}
Since the geodesic arcs $\zeta_n|_{[0,s_n]}$ converge to $\zeta(\cdot-s)|_{[0,s]}$ in the $C^\infty$ topology on every compact subinterval of $[0,s)$, we can fix $\epsilon\in(0,\tau_\gamma)$ small enough so that $a|_{\Upsilon_{n,\epsilon}}$ has a unique minimizer $\lambda_n$, which is a geodesic arc contained in the tubular neighborhood $N$, and we have
\begin{align*}
 \delta:=\inf_{n\in\N} \Big( a(\gamma|_{[\tau_\gamma-\epsilon,\tau_\gamma]}*\zeta_n|_{[0,\epsilon]})-a(\lambda_n) \Big)
 >0.
\end{align*}
The concatenation 
\[
\kappa_n:=\gamma|_{[0,\tau_\gamma-\epsilon]}*\lambda_n*\zeta_n|_{[\epsilon,t_n]}
\in
\Omega_{t_n+\tau_\gamma}
\] 
represents the same element of the fundamental group as $\gamma*\zeta_n$, and therefore
\[ [\kappa_n]= [\gamma][\zeta_n]\in GhG.\]
However, if $n$ is large enough so that $|a(\zeta_n)-a_{\infty}|<\delta$, we have
\begin{align*}
 a(\kappa_n)
 \leq 
 a(\zeta_n)-\delta 
 <
 a_\infty,
\end{align*}
which contradicts the fact that $\min a|_{\Omega_{t_n+\tau_\gamma}}\geq a_\infty$. This proves that $s_n\to\infty$, and an analogous argument implies that $t_n-s_n\to\infty$.

For each $s>0$, we have $[s_n-s,s_n+s]\subset[0,t_n]$ for $n$ large enough, and
\begin{align*}
 a(\zeta|_{[-s,s]}) 
 = 
 \lim_{n\to\infty}
 a(\zeta_n|_{[s_n-s,s_n+s]})
 \leq
 \lim_{n\to\infty}
 a(\zeta_n)
 =a_\infty.
\end{align*}
Therefore
\[
\int_{-\infty}^\infty F(\dot\zeta(t))\,dt \leq a_\infty.
\qedhere
\]
\end{proof}

\begin{proof}[Proof of Lemma~\ref{l:geodesic_with_finite_action} in the case $\pi_1(M,x_0)\cong\Z$]
Since $M$ has dimension at least two and fundamental group $\pi_1(M,x_0)\cong\Z$, there exists a minimal integer $k\geq1$ such that the higher homotopy group $\pi_{k+1}(M,x_0)\neq0$ is non-trivial. Indeed, otherwise any continuous map $\beta:S^1\to M$ representing a generator of $\pi_1(M,x_0)$ would be a homotopy equivalence, whereas a closed manifold of dimension at least two cannot be homotopy equivalent to a manifold of dimension one.

In order to simplify the notation, we can assume without loss of generality that the simple closed geodesic $\gamma$ has unit speed $\|\dot\gamma\|_{g}\equiv1$ and minimal period $\tau_\gamma=1$.
Let $\tau$ be a positive integer that we will fix soon. For each integer $n\geq1$, we set
$\gamma_n:=\gamma|_{[0,n\tau]}$, and consider the based and free loop spaces
\begin{align*}
\Omega_n & := \big\{ \zeta:[0,n\tau]\ttoup^{W^{1,2}} M\ \big|\ \zeta(0)=\zeta(n\tau)=x_0\big\},\\
\Lambda_n & := \big\{ \zeta:[0,n\tau]\ttoup^{W^{1,2}} M\ \big|\ \zeta(0)=\zeta(n\tau)\big\}.
\end{align*}
The concatenation with $\gamma_1$ defines a homotopy equivalence
\[i_n:\Omega_n \to \Omega_{n+1},\qquad
\zeta\mapsto\gamma_1*\zeta.\]
We denote by $j_n:\Omega_n\hookrightarrow\Lambda_n$ the inclusion.
A topological result of Bangert and Hingston \cite[Lemmas 1 and 2]{Bangert:1984aa} implies that, for a suitable value of the integer $\tau$, there exist non-trivial homotopy classes $h_n\in\pi_k(\Omega_n,\gamma_n)$ such that $h_{n+1}=i_{n*}h_n$, and their images $q_n:=j_{n*}h_n\in\pi_k(\Lambda_n,\gamma_n)$ are nontrivial as well. We fix a basepoint $z_0$ in the unit sphere $S^k$. The representatives of $h_n$ are continuous maps of pointed spaces of the form $\Gamma:(S^k,z_0)\to(\Omega_n,\gamma_n)$, and analogously the representatives of $q_n$ are continuous maps of pointed spaces of the form $\Gamma:(S^k,z_0)\to(\Lambda_n,\gamma_n)$. We define the min-max values
\begin{equation}
\label{e:min_max_values}
\begin{split}
 b_n & := \inf_{[\Gamma]=h_n} \max a\circ \Gamma,\\
 a_n & := \inf_{[\Gamma]=q_n} \max a\circ\Gamma. 
\end{split}
\end{equation}
Since $q_n=j_{n*}h_n$, we have 
\begin{align}
\label{e:bn>cn}
b_n\geq a_n.
\end{align}
For each representative $\Gamma$ of $h_n$, the composition $i_n\circ\Gamma$ is a representative of $h_{n+1}$, and since $a(\gamma_1)=0$, we have 
\[a(i_n\circ\Gamma(z))=a(\gamma_1)+a(\Gamma(z))=a(\Gamma(z)),
\qquad\forall z\in S^k.\]
This implies
\begin{align}
\label{e:bn>bn+1}
 b_n\geq b_{n+1}.
\end{align}

For each $\zeta\in\Omega_n$, we denote by $\overline\zeta\in\Omega_n$ the same geometric curve parametrized proportionally to arc-length, so that 
\[
\int_{0}^{tn\tau} \|\dot{\overline\zeta}(s)\|_{g}\,ds
=
t
\int_{0}^{n\tau} \|\dot{\zeta}(s)\|_{g}\,ds,
\qquad\forall t\in[0,1].
\] 
The map 
\[u_n:\Lambda_n\to\Lambda_n,\qquad u_n(\zeta)=\overline\zeta\] 
is continuous and homotopic to the identity (as it was proved by Anosov \cite[Theorem~2]{Anosov:1980aa}). Moreover, $u_n(\Omega_n)\subset\Omega_n$, and we have $a(\zeta)\geq a(u_n(\zeta))$ for all $\zeta\in\Lambda_n$. This shows that in the min-max expressions~\eqref{e:min_max_values} we can equivalently restrict the infima over maps that further satisfy $\Gamma=u_n\circ\Gamma$, that is, such that each loop $\Gamma(z)$ is parametrized proportionally to arc-length.

We fix a constant $c>0$ large enough so that 
\[F(v)\geq \tfrac14\|v\|_{g}^2-c,\qquad\forall v\in TM.\]
For each $\zeta\in\Lambda_n$ parametrized proportionally to arc-length, since 
\[
a(\zeta)\geq n\tau\big(\tfrac14\|\dot\zeta\|_{g}^2-c\big),
\] 
we have the a priori bound
\begin{align}
\label{e:apriori_bound_dot_zeta}
 \tfrac14\|\dot\zeta\|_{g}^2
 \leq
 \frac{a(\zeta)}{n\tau}+c
 \leq a(\zeta)+c.
\end{align}

Let $N\subset M$ be an open tubular neighborhood of $\gamma$.   
For each representative $\Gamma$ of $q_n$, there exists a point $z\in S^k$ such that the loop $\Gamma(z)$ is not entirely contained in $N$. Indeed, consider the free loop space 
\[ \Upsilon_n:=\big\{ \zeta:[0,n\tau]\ttoup^{W^{1,2}} N\ |\ \zeta(0)=\zeta(n\tau) \big\}. \]
Since $N$ is homotopy equivalent to a circle,  the evaluation map $\ev:\Lambda_n\to M$, $\ev(\zeta)=\zeta(0)$ restricts to a homotopy equivalence $\ev|_{\CC_n}:\CC_n\to N$, where $\CC_n$ is the connected component of $\Upsilon_n$ containing $\gamma_n$. In particular, $\ev$ induces an isomorphism 
\begin{align*}
\ev_*:\pi_k(\CC_n,\gamma_n)\ttoup^{\cong} \pi_k( N , x_0)
\cong
\left\{
  \begin{array}{@{}ll}
    \Z, & \mbox{if }k=1,\vspace{5pt} \\ 
    0, & \mbox{otherwise}. 
  \end{array}
\right. 
\end{align*}
Since $\ev\circ j_n\equiv x_0$ and $q_n=j_{n*}h_n$, we infer that $\ev_* q_n = (\ev\circ j_n)_* h_n  = 0$. Since the homotopy class $q_n$ is non-zero, no representative $\Gamma$ of $q_n$ can have its image contained in~$\CC_n$.

We claim that 
\begin{align*}
\inf_{n} a_n > 0.
\end{align*}
Indeed, let $N_0\subset M$ be another open tubular neighborhood of $\gamma$ whose closure is contained in $N$, and let $\rho>0$ be the minimum distance from points of $\partial N_0$ to points of $\partial N$. In particular, any smooth curve that crosses the shell $N\setminus N_0$ must have length at least $\rho$. Here, the distances and the lengths are measured with respect to the Riemannian metric $g$. Since $F$ is strictly positive outside $N_0$ and is coercive, we have
\begin{align*}
f:= \min\big\{F(v)\ \big|\ \pi(v)\in{M\setminus N_0}\big\} > 0.
\end{align*}
Let $\Gamma=u_n\circ\Gamma$ be a representative of $q_n$ that is not too far from being optimal, meaning that
\begin{align*}
 \max_{z\in S^k} a(\Gamma(z))\leq a_n+1.
\end{align*}
We know that there exists $z\in S^k$ such that the loop $\zeta:=\Gamma(z)$ is not entirely contained in $N$. If $\zeta$ intersects the smaller tubular neighborhood $N_0$, then there exists an interval $[t_0,t_1]\subset[0,n]$ such that $\zeta|_{[t_0,t_1]}$ has length $\|\dot\zeta\|_{g}(t_1-t_0)\geq\rho$ and is contained in $M\setminus N_0$; the a priori bound~\eqref{e:apriori_bound_dot_zeta}, together with~\eqref{e:bn>bn+1} and \eqref{e:bn>cn}, implies 
\begin{align*}
t_1-t_0
\geq
\frac{\rho}{\|\dot\zeta\|_{g}}
\geq
\frac{\rho}{\sqrt{4(b_1+1+c)}},
\end{align*}
and therefore
\begin{align*}
a(\zeta)
\geq
a(\zeta|_{[t_0,t_1]})
\geq 
(t_1-t_0)f
\geq
\frac{\rho f}{\sqrt{4(b_1+1+c)}}.
\end{align*}
If instead $\zeta$ does not intersect $N_0$, then
\begin{align*}
a(\zeta)\geq n \tau f \geq f.
\end{align*}

Standard variational methods imply that $a_n$ is a critical value of $a|_{\Lambda_n}$. Therefore $a_n=a(\zeta_n)$ for some closed geodesic $\zeta_n\in\Lambda_n$ contained in the connected component of $\gamma_n$. In particular, $\zeta_n$ is a geodesic loop such that $\dot\zeta_n(0)=\dot\zeta_n(n\tau)$, and therefore from now on we will see it as an $n\tau$-periodic geodesic $\zeta_n:\R\to M$.
Since $a_n>0$, we have that $\zeta_n$ is geometrically distinct from $\gamma$. We now consider the unit-sphere tangent bundle 
\begin{align*}
 SM = \big\{v\in TM\ \big|\ \|v\|_g=1\big\}.
\end{align*}
Since the geodesic flow on $SM$ is expansive near the hyperbolic periodic orbit $\dot\gamma$ (see, e.g., \cite[Cor.~5.3.5]{Fisher:0aa}), there exists a neighborhood $U\subset SM$ of $\dot\gamma$ such that, for each $n\geq1$, there exists $t_n\in[0,n\tau]$ such that 
\[\frac{\dot\zeta_n(t_n)}{\|\dot\zeta_n(t_n)\|_{g}}\not\in U.\]
Let $v_n:=\dot\zeta_n(t_n)$ be the corresponding tangent vector. The a priori bound~\eqref{e:apriori_bound_dot_zeta} implies that the sequence $\|v_n\|_{g}$ is uniformly bounded from above. We claim that the sequence $\|v_n\|_{g}$ is also uniformly bounded from below by a positive constant. Indeed, since the continuous Tonelli Lagrangian $F$ is strictly positive along the zero section of $TM$, there exists $r>0$ small enough so that
\begin{align*}
\delta:= \min_{\|v\|_{g}\leq r} F(v)>0.
\end{align*}
If we had $\|v_n\|_{g}\leq r$ for some integer $n> b_1/(\delta\tau)$, then we would get the contradiction 
\[a_n=a(\zeta_n)\geq\delta n\tau> b_1 \geq b_n \geq a_n.\]
Overall, we obtained a compact interval $[r_1,r_2]\subset(0,\infty)$ such that $r_1\leq\|v_n\|_{g}\leq r_2$ for all integers $n\geq1$. Therefore, up to extracting a subsequence, we have $v_n\to v_\infty$ and $a_n\to a_\infty$. If $\zeta:\R\to M$ is the geodesic such that $\dot\zeta(0)=v_\infty$, then $\zeta_n(t_n+\cdot)\to\zeta$ in the $C^\infty$-topology on every compact set. Since $F$ is non-negative, for each $s>0$ we have
\begin{align*}
a(\zeta|_{[-s,s]})
=
\lim_{n\to\infty}
a(\zeta_n|_{[t_n-s,t_n+s]})
\leq
\lim_{n\to\infty}
a_n
=
a_\infty,
\end{align*}
and therefore
\[
\int_{-\infty}^\infty F(\dot\zeta(t))\,dt \leq a_\infty.
\qedhere
\]
\end{proof}

\begin{proof}[Proof of Theorem~\ref{mt:main}]
By Theorem~\ref{t:homoclinic}, there exists a Riemannian metric $g_1$ arbitrarily $C^\infty$ close to $g$ such that $\gamma$ is a hyperbolic closed geodesic of $(M,g_1)$ with a homoclinic. A theorem due to Petroll \cite[Prop.~2.4]{Burns:2002aa} implies that there exists a Riemannian metric $g_2$ that is arbitrarily $C^\infty$ close to $g_1$ such that $\gamma$ is a hyperbolic closed geodesic of $(M,g_2)$ with a transverse homoclinic (if $M$ is a surface, the analogous theorem for $C^2$ perturbations of the Riemannian metric was proved independently by Donnay \cite{Donnay:1995aa}).
\end{proof}

We now provide a slight generalization of Theorem~\ref{mt:main} that essentially allows its assumptions to be verified by a finite cover of the considered closed Riemannian manifold.

\begin{Thm}\label{t:main_cover}
Let $p:M\to M_0$ be a finite covering map of a closed manifold of dimension at least two, and $g_0$  a Riemannian metric on $M_0$. If $(M,p^*g_0)$ has a minimal closed geodesic $\gamma$, then there exists a Riemannian metric $g$ arbitrarily $C^\infty$-close to $g_0$ such that $\gamma$ is a hyperbolic closed geodesic of $p^*g$, and $p(\gamma)$ is a hyperbolic simple closed geodesic of $g$ with a transverse homoclinic.
\end{Thm}

\begin{proof}
The proof is almost identical to the one of Theorem~\ref{t:homoclinic}, except for a few details.
Being $(p^*g_0,\sigma)$-minimal, the closed geodesic $\gamma:\R\to M$ is simple. Namely, $\gamma=\gamma(\tau_\gamma+\cdot)$ for some minimal period $\tau_\gamma>0$, and $\gamma|_{[0,\tau_\gamma)}$ is an injective map. We claim that, for each Deck transformation $\psi:M\to M$, the closed geodesic $\eta:=\psi\circ\gamma$ is either disjoint from $\gamma$ or is of the form $\eta=\gamma(\tau+\cdot)$ for some $\tau>0$. Indeed, assume by contradiction that there exist distinct $t_1,t_2\in[0,\tau_\gamma)$ such that $y:=\eta(t_1)=\gamma(t_2)$ but $\dot\eta(t_1)\neq\dot\gamma(t_2)$. Since both invariant measures $\mu_{\gamma}$ and $\mu_{\eta}$ are $(p^*g_0,\sigma)$-minimal, so is their average $\mu:=\tfrac12(\mu_\gamma+\mu_\eta)$. But the tangent vectors $\dot\eta(t_1),\dot\gamma(t_2)\in\supp(\mu)$ are based at the same point $y$, and therefore $\pi|_{\supp(\mu)}$ is not injective, contradicting Mather's graph theorem \cite[Theorem~2]{Mather:1991aa}.

This implies that $\gamma_0:=p\circ\gamma$ is also a simple closed geodesic for the Riemannian metric $g_0$ (although $\tau_\gamma$ may be a multiple of the minimal period of $\gamma_0$). 
We can now carry out word by word the proof of Theorem~\ref{t:homoclinic}, with the only difference that here we  apply Proposition~\ref{p:make_hyperbolic} to the simple closed geodesic $\gamma_0$ in the base manifold, and therefore we obtain the conformal factor $\rho$ of the form $\rho=\rho_0\circ p$, where $\rho_0:M_0\to[0,\infty)$ is a suitable function vanishing on $\gamma_0$ and strictly positive outside $\gamma_0$. We end up with a Riemannian metric $g_1=e^{\rho_0}g_0$ on $M_0$ arbitrarily $C^\infty$-close to $g_0$ such that $\gamma_0$ is a hyperbolic simple closed geodesic for the Riemannian metric $g_1$, and therefore $\gamma$ is a hyperbolic simple closed geodesic for the Riemannian metric $p^*g_1$. Instead of vanishing only on $\gamma$ as in the proof of Theorem~\ref{t:homoclinic}, the function $\rho$ here vanishes on all the images of $\gamma$ under Deck transformations. Therefore, at the end of the proof, instead of a homoclinic to $\gamma$, we obtain a heteroclinic $\zeta$ from $\gamma$ to $\psi\circ\gamma$, for some Deck transformation $\psi:M\to M$. Nevertheless, its base projection $\zeta_0:=p\circ\zeta$ is a homoclinic to $\gamma_0$. We then conclude the proof by applying Petroll's theorem \cite[Prop.~2.4]{Burns:2002aa} to $\gamma_0$, obtaining a Riemannian metric $g_2$ arbitrarily $C^\infty$ close to $g_1$ with respect to which $\gamma_0$ is a hyperbolic closed geodesic with a transverse homoclinic.
\end{proof}

For a group $G$, we denote its derived series by $G_n$, for $n\geq0$. These groups are defined inductively as $G_0=G$ and $G_{n+1}=[G_n,G_n]$, where this latter group is the commutator subgroup of $G_n$. We denote by $|G:G_n|$ the index of the derived subgroup $G_n$. As in Section~\ref{s:intro}, we denote by $\GG^k(M)$ the space of smooth Riemannian metrics on a closed manifold $M$, endowed with the $C^k$ topology.

\begin{Cor}\label{c:multiplicity}
Let $M$ be a closed connected manifold of dimension at least two such that the index $|\pi_1(M):\pi_1(M)_n|$ is infinite for some $n\geq1$. Then, for each $2\leq k\leq\infty$, there exists an open and dense subset of $\GG^k(M)$ such that every Riemannian metric therein admits infinitely many closed geodesics of arbitrarily large length.
\end{Cor}

\begin{Remark}
We recall that $H_1(M;\Z)\cong\pi_1(M)/\pi_1(M)_1$ according to Hurewicz theorem. 
Since $M$ is a closed manifold, the homology group $H_1(M;\Z)$ is finitely generated. This implies that the index $|\pi_1(M):\pi_1(M)_1|$ is infinite if and only if the first Betti number 
$\rank (H_1(M;\Z))$
is non-zero. Therefore, Corollary~\ref{mc:multiplicity} directly follows from Corollary~\ref{c:multiplicity}.
\end{Remark}

\begin{proof}[Proof of Corollary~\ref{c:multiplicity}]
Let $G_n$, $n\geq0$, be the derived series of the fundamental group $\pi_1(M)$. We have an associated sequence of normal covering spaces 
\[...\to M_2\to M_1\to M_0=M\] 
with fundamental groups $\pi_1(M_n)=G_n$. The quotient $G_{n}/G_{n+1}$ is the group of Deck transformations of the covering $M_{n+1}\to M_n$. 
Let $n\geq0$ be the minimal integer such that $G_n/G_{n+1}$ is infinite, which exists by the assumption of the corollary. Notice that $p:M_n\to M$ is a finite covering. Therefore $M_n$ is a closed manifold, and by our choice of $n$ it has infinite homology group $H_1(M_n;\Z)\cong G_{n}/G_{n+1}$. Since $H_1(M_n;\Z)$ is finitely generated, the first Betti number $\rank(H_1(M_n;\Z))$ is non-zero.

We fix an integer $k\geq2$, and denote by $\mathcal{I}\subset\GG^k(M)$ the subspace of those Riemannian metrics $g$ on $M$ having closed geodesics of arbitrarily large length. We need to show that $\mathcal{I}$ contains an open and dense subset of $\GG^k(M)$. We denote by $\mathcal{H}\subset\GG^k(M)$ the open subspace of those Riemannian metrics $g$ on $M$ having a hyperbolic closed geodesic with a transverse homoclinic. 
As we already mentioned in Section~\ref{s:intro}, classical results from hyperbolic dynamics imply that
\begin{align*}
 \mathcal{H}\subset\mathcal{I}.
\end{align*}
We denote by $\mathcal{A}\subset\GG^k(M)$ the subspace of those Riemannian metrics $g$ on $M$ such that $(M_n,p^*g)$ admits a minimal closed geodesic. 
Theorems~\ref{t:aperiodic} and~\ref{t:main_cover} imply that
\begin{align*}
\mathcal{A}\subset\overline{\mathcal{H}},
\qquad
\GG^k(M)\setminus\mathcal{A}\subset\mathcal{I}.
\end{align*}
We define $\mathcal{B}:=\GG^k(M)\setminus\overline{\mathcal{A}}$. We have
\begin{align*}
 \overline{ \mathcal{H}\cup\mathcal{B} }
 =
 \overline{ \mathcal{H}} \cup \overline{\mathcal{B} }
 \supseteq
 \overline{\mathcal{A}} \cup (\GG^k(M)\setminus\overline{\mathcal{A}})
 =
 \GG^k(M).
\end{align*}
We thus have an open and dense subset $\mathcal{H}\cup\mathcal{B}$ of $\GG^k(M)$ that is contained in~$\mathcal{I}$.
\end{proof}

\appendix

\section{Making closed geodesics hyperbolic}
\label{s:hyperbolic}

In this appendix we shall provide a proof of the following statement, which is employed in the proof of Theorem~\ref{mt:main}. We recall that a closed geodesic $\gamma:\R\to M$ of minimal period $\tau_\gamma>0$ is called \emph{simple} when $\gamma|_{[0,\tau_\gamma)}$ is an injective map.

\begin{Prop}\label{p:make_hyperbolic}
Let $\gamma$ be a simple closed geodesic without conjugate points in a closed Riemannian manifold $(M,g)$ of dimension at least two. Then $\gamma$ is a hyperbolic closed geodesic with respect to the conformal Riemannian metric $e^\rho g$, for any smooth function $\rho:M\to[0,\infty)$ such that $\rho(x)=0$ and $d^2\rho(x)[v,v]>0$ for all $x\in\gamma$ and $v\in T_xM\setminus\{0\}$ orthogonal to $\gamma$.
\end{Prop}

Proposition~\ref{p:make_hyperbolic} guarantees that, given a simple closed geodesic $\gamma$ without conjugate points, there exists an arbitrarily $C^\infty$-small conformal perturbation of the Riemannian metric that makes $\gamma$ hyperbolic. An analogous result for perturbations with potentials of Tonelli Hamiltonians was proved in \cite{Contreras:1999ab}.

\subsection{Green spaces}\label{s:Green_spaces}
Let us recall some basic facts from geodesic dynamics (for the details, we refer the reader to, e.g., \cite{Paternain:1999aa,Knieper:2002aa,Guillarmou:2025aa}).
Let $(M,g)$ be a closed Riemannian manifold of dimension at least two. We denote its sphere tangent bundle of radius $r>0$ by 
\[S^r\!M=\big\{v\in TM\ \big|\ \|v\|_g=r\big\},\] 
and the geodesic flow  by $\phi^t:S^rM\to S^rM$. The orbits of $\phi^t$ are of the form $\phi^t(\dot\gamma(0))=\dot\gamma(t)$, where $\gamma:\R\to M$ is a geodesic parametrized with speed $\|\dot\gamma\|_g\equiv r$. Without loss of generality, throughout this section we shall always assume that all geodesics are parametrized with speed $r=1$, and simply write $SM=S^1\!M$.

Let $\gamma:\R\to M$ be a geodesic, so that $\dot\gamma(t)=\phi^t(v)$ for $v=\dot\gamma(0)\in S_xM$.
We introduce the vector subspace $Z:=d\pi(v)^{-1}\langle v\rangle^\bot\subset T_v(SM)$, where $\pi:SM\to M$ is the base projection, and $\langle v\rangle^\bot\subset T_xM$ is the orthogonal complement to $v$. As it is common, we identify
\begin{equation}
\label{e:identification_TSM}
\begin{split}
Z & \equiv \langle v\rangle^\bot\times \langle v\rangle^\bot,\\
 \dot J(0) & \mapsto (J(0),\nabla_t J|_{t=0}),
\end{split}
\end{equation}
where $J$ is any Jacobi field orthogonal to $\gamma$, and $\nabla_t J$ is its covariant derivative with respect to the Levi-Civita connection.
We denote by $V:=\ker(d\pi)$ the vertical sub-bundle of $T(SM)$. By the identification~\eqref{e:identification_TSM}, we have $V_v\equiv\{0\}\times  \langle v\rangle^\bot$.

We assume that $\gamma$ is without conjugate points, which is equivalent to  
\[V_{v}\cap d\phi^{-t}(v)V_{\phi^t(v)}=\{0\},\qquad\forall t\in\R\setminus\{0\}.\]
We define the vector subspaces
\begin{align}
\label{e:Green_t}
 G_t:= d\phi^{-t}(\phi^{t}(v))V_{\phi^{t}(v)} \subset Z.
\end{align}
Since $\gamma$ is without conjugate points, for each $t\neq0$ the vector subspace $G_t$ is transverse to the vertical $V_v$. Via~\eqref{e:identification_TSM}, we shall always see $G_t$ as a vector subspace of $\langle v\rangle^\bot\times \langle v\rangle^\bot$, and the transversality with $V_v$ implies that $G_t$ is a graph over the horizontal $\langle v\rangle^\bot\times\{0\}$. More precisely, there exist linear symmetric endomorphisms $A_t:\langle v\rangle^\bot\to \langle v\rangle^\bot$, depending smoothly on $t\in\R\setminus\{0\}$, such that
$G_t\equiv\mathrm{graph}(A_t)$.
The associated quadratic forms $Q_t(w)=g(A_tw,w)$ are monotone increasing in $t\in\R\setminus\{0\}$, and we have $Q_t\leq Q_{-s}$ for all $s,t>0$. Therefore, the limits
\begin{align*}
 A_{\pm}:=\lim_{t\to\pm\infty} A_t,
 \qquad
 G_{\pm}:=\lim_{t\to\pm\infty} G_t,
\end{align*}
exist, and we have 
$G_\pm\equiv\mathrm{graph}(A_\pm)$.
The associated quadratic forms $Q_\pm(w)=g(A_\pm w,w)$ satisfy $Q_+\leq Q_-$.
Pushing forward $G_\pm$ with the linearized geodesic flow $d\phi^t$, we obtain the so-called Green bundles of $\gamma$, which are well defined even if $\gamma$ is a closed geodesic. For our purposes, we will only need the vector spaces $G_\pm$, which we will call Green spaces. We will refer to the linear maps $A_\pm$ as to Green endomorphisms.

Let us now assume that $\gamma$ is closed. We recall that  $\gamma$ is said to be hyperbolic when $\dot\gamma$ is a hyperbolic periodic orbit of the geodesic flow $\phi^t$.
The proof of Proposition~\ref{p:make_hyperbolic} will require the following special case of a theorem due to Eberlein \cite{Eberlein:1973aa} (the general statement actually holds for arbitrary compact invariant subsets of $\phi^t$ without conjugate points). 

\begin{Thm}[Eberlein]
\label{t:Eberlein}
A closed geodesic without conjugate points is hyperbolic if and only if its Green spaces satisfy $G_-\cap G_+=\{0\}$.
\hfill\qed
\end{Thm}

\subsection{The index form}
Let $\gamma:\R\to M$ be a geodesic without conjugate points. We denote by $\XX$ the space of orthogonal vector fields $Y:\R\to TM$ along $\gamma$, where orthogonal means $g(\dot\gamma,Y)\equiv0$. We denote by $\JJ\subset\XX$ the subspace of orthogonal Jacobi fields, which are those $J\in\XX$ that satisfy the Jacobi equation 
\[\nabla_t^2 J+R(J,\dot\gamma)\dot\gamma=0;\] 
equivalently, they are those vector fields $J:\R\to TM$ along $\gamma$ such that 
\[\dot J(t)=d\phi^t(v)w,\]
where $v=\dot\gamma(0)$ and $w\in Z$ (with the notation of the previous subsection).

For each $a<b$, we consider the index form of the geodesic arc $\gamma|_{[a,b]}$, which is the quadratic form
\begin{align*}
 h_{[a,b]}(Y)
 =
 \int_a^b \Big(\|\nabla_t Y\|_{g}^2-g(R(Y,\dot\gamma)\dot\gamma,Y)\Big)\,dt,
 \qquad\forall Y\in\XX.
\end{align*}
We shall need two elementary properties of the index form:
\begin{itemize}
\item[(i)] For each $J\in\JJ$, we have
\begin{align*}
 h_{[a,b]}(J)=g(\nabla_t J|_{t=b},J(b))-g(\nabla_t J|_{t=a},J(a)). 
\end{align*}

\item[(ii)] For each $J\in\JJ$ and $Y\in\XX$ such that $J(a)=Y(a)$ and $J(b)=Y(b)$, we have
\begin{align*}
 h_{[a,b]}(J)\leq h_{[a,b]}(Y). 
\end{align*}
\end{itemize}
For each $\tau\in\R\setminus\{0\}$, we set
\begin{align*}
h_\tau(Y)
:=
\left\{
  \begin{array}{@{}cc}
    h_{[0,\tau]}(Y), & \mbox{if }\tau>0, \vspace{5pt}\\ 
    h_{[\tau,0]}(Y), & \mbox{if }\tau<0. 
  \end{array}
\right. 
\end{align*}
We now employ Eberlein's theorem, together with the index form, to prove the perturbation result stated at the beginning of the section.

\begin{proof}[Proof of Proposition~\ref{p:make_hyperbolic}]
We set $\tilde g:=e^\rho g$. The Riemannian metrics $g$ and $\tilde g$ define associated Levi-Civita connections $\nabla$ and $\tilde\nabla$ and Riemann tensors $R$ and $\tilde R$. Along $\gamma$, since $\rho$ and $d\rho$ vanish identically, we have $\nabla=\tilde\nabla$ and
\begin{equation}
\label{e:R_vs_tildeR}
\begin{split}
 g(\tilde R(w,\dot\gamma(t))\dot\gamma(t),w)-g(R(w,\dot\gamma(t))\dot\gamma(t),w)
=
-\tfrac12 d^2\rho(\gamma(t))[w,w],\\
\forall w\in\langle\dot\gamma(t)\rangle^\bot.
\end{split}
\end{equation}
By our assumptions on $\rho$, there exists a constant $\delta>0$ such that 
\begin{equation}
\label{e:rho_convex}
\tfrac12 d^2\rho(\gamma(t))[w,w]\geq \delta\,\|w\|_g^2,
\qquad\forall w\in\langle\dot\gamma(t)\rangle^\bot.
\end{equation}
We set $v:=\dot\gamma(0)$, and first consider the Riemannian objects associated with $g$. For each $w\in\langle v\rangle^\bot$ and $\tau\neq0$, we denote by $J_{\tau,w}$ the Jacobi field along $\gamma$ such that $J_{\tau,w}(0)=w$ and $J_{\tau,w}(\tau)=0$. Notice that $\nabla_t J_{\tau,w}|_{t=0}=A_{\tau}w$, where $A_\tau$ is the symmetric endomorphism of $\langle v\rangle^\bot$ converging to the Green endomorphisms $A_\pm$ as $\tau\to\pm\infty$. By property (i) above, the index forms $h_\tau$ of $\gamma$ with respect to $g$ satisfy
\begin{align*}
 h_\tau(J_{\tau,w}) 
 = 
 \left\{
   \begin{array}{@{}ll}
    - g(A_\tau w,w), & \mbox{if }\tau>0, \vspace{5pt} \\ 
    g(A_\tau w,w), & \mbox{if }\tau<0. 
  \end{array}
 \right.
\end{align*}
We denote with a tilde the analogous Riemannian objects with respect to $\tilde g$, which satisfy analogous properties. By \eqref{e:R_vs_tildeR}, \eqref{e:rho_convex}, and property (ii) of the index form, for each $\tau\geq1$ we have
\begin{equation}
\label{e:Green_comparison}
\begin{split}
 g(A_\tau w,w)
&=
- h_\tau(J_{\tau,w})
\geq
- h_\tau(\tilde J_{\tau,w})\\
&=
- \tilde h_\tau(\tilde J_{\tau,w}) 
+ 
\int_0^\tau \tfrac12 d^2\rho(\gamma)[\tilde J_{\tau,w},\tilde J_{\tau,w}]\,dt
\\
&\geq
- \tilde h_\tau(\tilde J_{\tau,w}) + \delta\int_0^1 \|\tilde J_{\tau,w}\|_g^2\,dt\\
& =
g(\tilde A_\tau w,w) + \delta\int_0^1 \|\tilde J_{\tau,w}\|_g^2\,dt.
\end{split}
\end{equation}
As $\tau\to\infty$, we have $\tilde J_{\tau,w}\to\tilde J_w$, where $\tilde J_w$ is the Jacobi field such that $\tilde J_w(0)=w$ and $\nabla_t \tilde J_w|_{t=0}=\tilde A_+w$. We set
\begin{align*}
 \epsilon:=\delta \min_{w} \|w\|_g^{-2}\int_0^1 \|\tilde J_{w}\|_g^2\,dt >0,
\end{align*}
where the minimum ranges over all $w\in\langle v\rangle^\bot\setminus\{0\}$.
By taking the limit for $\tau\to\infty$ in~\eqref{e:Green_comparison}, we infer
\begin{align}
\label{e:comparison_A_1}
  g(A_+ w,w)
  \geq
  g(\tilde A_+ w,w)+ \epsilon\|w\|_g^2.
\end{align}
Analogously, for each $\tau<0$, we have
\begin{equation*}
 g(A_{\tau} w,w)
=
 h_\tau(J_{\tau,w})
\leq
 h_\tau(\tilde J_{\tau,w})
\leq
\tilde h_\tau(\tilde J_{\tau,w})
=
g(\tilde A_{\tau} w,w),
\end{equation*}
and by taking the limit for $\tau\to-\infty$ we infer
\begin{align}\label{e:comparison_A_2}
 g(A_{-} w,w)\leq g(\tilde A_{-} w,w).
\end{align}
The inequalities~\eqref{e:comparison_A_1} and~\eqref{e:comparison_A_2}, together with  $g(A_+ w,w)\leq g(A_- w,w)$ mentioned in the previous subsection, imply
\begin{align*}
g(\tilde A_+ w,w)+\epsilon\|w\|_g^2
\leq
g(\tilde A_- w,w),\qquad\forall w\in\langle v\rangle^\bot.
\end{align*}
Therefore the Green spaces $\tilde G_+\equiv\mathrm{graph}(\tilde A_+)$ and $\tilde G_-\equiv\mathrm{graph}(\tilde A_-)$ have trivial intersection, and Theorem~\ref{t:Eberlein} implies that $\gamma$ is a hyperbolic closed geodesic for~$\tilde g$.
\end{proof}

\bibliographystyle{amsalpha}
\bibliography{biblio}

\def\cprime{$'$} \def\cprime{$'$} \def\cprime{$'$} \def\cprime{$'$}
\begin{bibdiv}
\begin{biblist}

\bib{Anosov:1980aa}{article}{
      author={Anosov, D.~V.},
       title={Some homotopies in a space of closed curves},
        date={1980},
        ISSN={0373-2436},
     journal={Izv. Akad. Nauk SSSR Ser. Mat.},
      volume={44},
      number={6},
       pages={1219\ndash 1254, 38},
  url={https://mathscinet-ams-org.acces.bibliotheque-diderot.fr/mathscinet-getitem?mr=603576},
      review={\MR{603576}},
}

\bib{Bangert:1990aa}{article}{
      author={Bangert, Victor},
       title={Minimal geodesics},
        date={1990},
        ISSN={0143-3857},
     journal={Ergodic Theory Dynam. Systems},
      volume={10},
      number={2},
       pages={263\ndash 286},
  url={https://doi-org.acces.bibliotheque-diderot.fr/10.1017/S014338570000554X},
      review={\MR{1062758}},
}

\bib{Bangert:1993aa}{article}{
      author={Bangert, Victor},
       title={On the existence of closed geodesics on two-spheres},
        date={1993},
        ISSN={0129-167X},
     journal={Internat. J. Math.},
      volume={4},
      number={1},
       pages={1\ndash 10},
  url={https://doi-org.acces.bibliotheque-diderot.fr/10.1142/S0129167X93000029},
      review={\MR{1209957}},
}

\bib{Bangert:1984aa}{article}{
      author={Bangert, Victor},
      author={Hingston, Nancy},
       title={Closed geodesics on manifolds with infinite abelian fundamental
  group},
        date={1984},
        ISSN={0022-040X},
     journal={J. Differential Geom.},
      volume={19},
      number={2},
       pages={277\ndash 282},
  url={http://projecteuclid.org.acces.bibliotheque-diderot.fr/euclid.jdg/1214438679},
      review={\MR{755226}},
}

\bib{Bolotin:1995aa}{incollection}{
      author={Bolotin, Sergey~V.},
       title={Homoclinic orbits in invariant tori of {H}amiltonian systems},
        date={1995},
   booktitle={Dynamical systems in classical mechanics},
      series={Amer. Math. Soc. Transl. Ser. 2},
      volume={168},
   publisher={Amer. Math. Soc., Providence, RI},
       pages={21\ndash 90},
         url={https://doi.org/10.1090/trans2/168/03},
      review={\MR{1351032}},
}

\bib{Bolotin:1995ab}{inproceedings}{
      author={Bolotin, Sergey~V.},
       title={Invariant sets of {H}amiltonian systems and variational methods},
        date={1995},
   booktitle={Proceedings of the {I}nternational {C}ongress of
  {M}athematicians, {V}ol. 1, 2 ({Z}\"{u}rich, 1994)},
   publisher={Birkh\"{a}user, Basel},
       pages={1169\ndash 1178},
         url={https://mathscinet.ams.org/mathscinet-getitem?mr=1404017},
      review={\MR{1404017}},
}

\bib{Bolotin:1997aa}{article}{
      author={Bolotin, Sergey},
       title={Homoclinic trajectories of invariant sets of {H}amiltonian
  systems},
        date={1997},
        ISSN={1021-9722},
     journal={NoDEA Nonlinear Differential Equations Appl.},
      volume={4},
      number={3},
       pages={359\ndash 389},
         url={https://doi.org/10.1007/s000300050020},
      review={\MR{1458533}},
}

\bib{Ballmann:1981aa}{article}{
      author={Ballmann, W.},
      author={Thorbergsson, G.},
      author={Ziller, W.},
       title={Closed geodesics and the fundamental group},
        date={1981},
        ISSN={0012-7094},
     journal={Duke Math. J.},
      volume={48},
      number={3},
       pages={585\ndash 588},
  url={http://projecteuclid.org.acces.bibliotheque-diderot.fr/euclid.dmj/1077314782},
      review={\MR{630586}},
}

\bib{Burns:2002aa}{article}{
      author={Burns, Keith},
      author={Weiss, Howard},
       title={Spheres with positive curvature and nearly dense orbits for the
  geodesic flow},
        date={2002},
        ISSN={0143-3857},
     journal={Ergodic Theory Dynam. Systems},
      volume={22},
      number={2},
       pages={329\ndash 348},
  url={https://doi-org.acces.bibliotheque-diderot.fr/10.1017/S0143385702000172},
      review={\MR{1898795}},
}

\bib{Carneiro:1995aa}{article}{
      author={Carneiro, M. J.~Dias},
       title={On minimizing measures of the action of autonomous
  {L}agrangians},
        date={1995},
        ISSN={0951-7715},
     journal={Nonlinearity},
      volume={8},
      number={6},
       pages={1077\ndash 1085},
  url={http://stacks.iop.org.acces.bibliotheque-diderot.fr/0951-7715/8/1077},
      review={\MR{1363400}},
}

\bib{Contreras:1999ab}{article}{
      author={Contreras, Gonzalo},
      author={Iturriaga, Renato},
       title={Convex {H}amiltonians without conjugate points},
        date={1999},
        ISSN={0143-3857},
     journal={Ergodic Theory Dynam. Systems},
      volume={19},
      number={4},
       pages={901\ndash 952},
  url={https://doi-org.acces.bibliotheque-diderot.fr/10.1017/S014338579913387X},
      review={\MR{1709426}},
}

\bib{Contreras:1999aa}{book}{
      author={Contreras, Gonzalo},
      author={Iturriaga, Renato},
       title={Global minimizers of autonomous {L}agrangians},
      series={22$^{\rm o}$ Col\'{o}quio Brasileiro de Matem\'{a}tica},
   publisher={Instituto de Matem\'{a}tica Pura e Aplicada (IMPA), Rio de
  Janeiro},
        date={1999},
        ISBN={85-244-0151-6},
  url={https://mathscinet-ams-org.acces.bibliotheque-diderot.fr/mathscinet-getitem?mr=1720372},
      review={\MR{1720372}},
}

\bib{Contreras:2010aa}{article}{
      author={Contreras, Gonzalo},
       title={Geodesic flows with positive topological entropy, twist maps and
  hyperbolicity},
        date={2010},
        ISSN={0003-486X},
     journal={Ann. of Math. (2)},
      volume={172},
      number={2},
       pages={761\ndash 808},
  url={https://doi-org.acces.bibliotheque-diderot.fr/10.4007/annals.2010.172.761},
      review={\MR{2680482}},
}

\bib{Contreras:2002aa}{article}{
      author={Contreras, Gonzalo},
      author={Paternain, Gabriel~P.},
       title={Connecting orbits between static classes for generic {L}agrangian
  systems},
        date={2002},
        ISSN={0040-9383},
     journal={Topology},
      volume={41},
      number={4},
       pages={645\ndash 666},
  url={https://doi-org.acces.bibliotheque-diderot.fr/10.1016/S0040-9383(00)00042-2},
      review={\MR{1905833}},
}

\bib{Donnay:1995aa}{incollection}{
      author={Donnay, Victor~J.},
       title={Transverse homoclinic connections for geodesic flows},
        date={1995},
   booktitle={Hamiltonian dynamical systems ({C}incinnati, {OH}, 1992)},
      series={IMA Vol. Math. Appl.},
      volume={63},
   publisher={Springer, New York},
       pages={115\ndash 125},
         url={https://doi.org/10.1007/978-1-4613-8448-9_7},
      review={\MR{1350307}},
}

\bib{Eberlein:1973aa}{article}{
      author={Eberlein, Patrick},
       title={When is a geodesic flow of {A}nosov type? {I},{II}},
        date={1973},
        ISSN={0022-040X},
     journal={J. Differential Geometry},
      volume={8},
       pages={437\ndash 463; ibid. 8 (1973), 565\ndash 577},
  url={http://projecteuclid.org.acces.bibliotheque-diderot.fr/euclid.jdg/1214431801},
      review={\MR{380891}},
}

\bib{Fathi:Weak_KAM_theorem_in_Lagrangian_dynamics}{book}{
      author={Fathi, Albert},
       title={Weak {KAM} theorem in {L}agrangian dynamics},
   publisher={unpublished book},
        date={2008},
}

\bib{Fisher:0aa}{book}{
      author={Fisher, Todd},
      author={Hasselblatt, Boris},
       title={Hyperbolic flows},
      series={Zurich Lectures in Advanced Mathematics},
   publisher={EMS Publishing House, Berlin},
        date={2019},
        ISBN={978-3-03719-200-9},
         url={https://doi-org.acces.bibliotheque-diderot.fr/10.4171/200},
      review={\MR{3972204}},
}

\bib{Franks:1992aa}{article}{
      author={Franks, John},
       title={Geodesics on {$S^2$} and periodic points of annulus
  homeomorphisms},
        date={1992},
        ISSN={0020-9910},
     journal={Invent. Math.},
      volume={108},
      number={2},
       pages={403\ndash 418},
  url={https://doi-org.acces.bibliotheque-diderot.fr/10.1007/BF02100612},
      review={\MR{1161099}},
}

\bib{Fathi:2004aa}{article}{
      author={Fathi, Albert},
      author={Siconolfi, Antonio},
       title={Existence of {$C^1$} critical subsolutions of the
  {H}amilton-{J}acobi equation},
        date={2004},
        ISSN={0020-9910},
     journal={Invent. Math.},
      volume={155},
      number={2},
       pages={363\ndash 388},
         url={https://doi.org/10.1007/s00222-003-0323-6},
      review={\MR{2031431}},
}

\bib{Guillarmou:2025aa}{book}{
      author={Guillarmou, Colin},
      author={Mazzucchelli, Marco},
       title={An introduction to geometric inverse problems},
   publisher={monograph, preliminary version available on the authors'
  websites},
        date={2025},
}

\bib{Gromoll:1969aa}{article}{
      author={Gromoll, Detlef},
      author={Meyer, Wolfgang},
       title={Periodic geodesics on compact riemannian manifolds},
        date={1969},
        ISSN={0022-040X},
     journal={J. Differential Geometry},
      volume={3},
       pages={493\ndash 510},
  url={http://projecteuclid.org.acces.bibliotheque-diderot.fr/euclid.jdg/1214429070},
      review={\MR{264551}},
}

\bib{Hingston:1984aa}{article}{
      author={Hingston, N.},
       title={Equivariant {M}orse theory and closed geodesics},
        date={1984},
        ISSN={0022-040X},
     journal={J. Differential Geom.},
      volume={19},
      number={1},
       pages={85\ndash 116},
  url={http://projecteuclid.org.acces.bibliotheque-diderot.fr/euclid.jdg/1214438424},
      review={\MR{739783}},
}

\bib{Hingston:1993aa}{article}{
      author={Hingston, Nancy},
       title={On the growth of the number of closed geodesics on the
  two-sphere},
        date={1993},
        ISSN={1073-7928},
     journal={Internat. Math. Res. Notices},
      number={9},
       pages={253\ndash 262},
  url={https://doi-org.acces.bibliotheque-diderot.fr/10.1155/S1073792893000285},
      review={\MR{1240637}},
}

\bib{Knieper:2002aa}{incollection}{
      author={Knieper, Gerhard},
       title={Hyperbolic dynamics and {R}iemannian geometry},
        date={2002},
   booktitle={Handbook of dynamical systems, {V}ol. 1{A}},
   publisher={North-Holland, Amsterdam},
       pages={453\ndash 545},
  url={https://doi-org.acces.bibliotheque-diderot.fr/10.1016/S1874-575X(02)80008-X},
      review={\MR{1928523}},
}

\bib{Mather:1991aa}{article}{
      author={Mather, John~N.},
       title={Action minimizing invariant measures for positive definite
  {L}agrangian systems},
        date={1991},
        ISSN={0025-5874},
     journal={Math. Z.},
      volume={207},
      number={2},
       pages={169\ndash 207},
         url={https://doi.org/10.1007/BF02571383},
      review={\MR{1109661}},
}

\bib{Mane:1996aa}{article}{
      author={Ma\~{n}\'{e}, Ricardo},
       title={Generic properties and problems of minimizing measures of
  {L}agrangian systems},
        date={1996},
        ISSN={0951-7715},
     journal={Nonlinearity},
      volume={9},
      number={2},
       pages={273\ndash 310},
         url={https://doi.org/10.1088/0951-7715/9/2/002},
      review={\MR{1384478}},
}

\bib{Paternain:1999aa}{book}{
      author={Paternain, Gabriel~P.},
       title={Geodesic flows},
      series={Progress in Mathematics},
   publisher={Birkh\"{a}user Boston, Inc., Boston, MA},
        date={1999},
      volume={180},
        ISBN={0-8176-4144-0},
  url={https://doi-org.acces.bibliotheque-diderot.fr/10.1007/978-1-4612-1600-1},
      review={\MR{1712465}},
}

\bib{Rademacher:1989aa}{article}{
      author={Rademacher, Hans-Bert},
       title={On the average indices of closed geodesics},
        date={1989},
        ISSN={0022-040X},
     journal={J. Differential Geom.},
      volume={29},
      number={1},
       pages={65\ndash 83},
  url={http://projecteuclid.org.acces.bibliotheque-diderot.fr/euclid.jdg/1214442633},
      review={\MR{978076}},
}

\bib{Rademacher:2022aa}{article}{
      author={Rademacher, Hans-Bert},
      author={Taimanov, Iskander~A.},
       title={Closed geodesics on connected sums and 3-manifolds},
        date={2022},
        ISSN={0022-040X},
     journal={J. Differential Geom.},
      volume={120},
      number={3},
       pages={557\ndash 573},
  url={https://doi-org.acces.bibliotheque-diderot.fr/10.4310/jdg/1649953350},
      review={\MR{4408292}},
}

\bib{Sorrentino:2015aa}{book}{
      author={Sorrentino, Alfonso},
       title={Action-minimizing methods in {H}amiltonian dynamics: an
  introduction to {A}ubry-{M}ather theory},
      series={Mathematical Notes},
   publisher={Princeton University Press, Princeton, NJ},
        date={2015},
      volume={50},
        ISBN={978-0-691-16450-2},
  url={https://doi-org.acces.bibliotheque-diderot.fr/10.1515/9781400866618},
      review={\MR{3330134}},
}

\bib{Sullivan:1975aa}{article}{
      author={Sullivan, Dennis},
      author={Vigu\'{e}-Poirrier, Micheline},
       title={Sur l'existence d'une infinit\'{e} de g\'{e}od\'{e}siques
  p\'{e}riodiques sur une vari\'{e}t\'{e} riemannienne compacte},
        date={1975},
        ISSN={0151-0509},
     journal={C. R. Acad. Sci. Paris S\'{e}r. A-B},
      volume={281},
      number={9},
       pages={Aii, A289\ndash A291},
  url={https://mathscinet-ams-org.acces.bibliotheque-diderot.fr/mathscinet-getitem?mr=400298},
      review={\MR{400298}},
}

\bib{Shelukhin:2023aa}{article}{
      author={Shelukhin, Egor},
      author={Zhang, Jun},
       title={Remark on non-contractible closed geodesics and homotopy groups},
        date={2023},
     journal={arXiv:2307.13877},
}

\bib{Taimanov:1985aa}{article}{
      author={Ta\u{\i}manov, I.~A.},
       title={Closed geodesics on non-simply-connected manifolds},
        date={1985},
        ISSN={0042-1316},
     journal={Uspekhi Mat. Nauk},
      volume={40},
      number={6(246)},
       pages={157\ndash 158},
  url={https://mathscinet-ams-org.acces.bibliotheque-diderot.fr/mathscinet-getitem?mr=815506},
      review={\MR{815506}},
}

\bib{Vigue-Poirrier:1976aa}{article}{
      author={Vigu\'{e}-Poirrier, Micheline},
      author={Sullivan, Dennis},
       title={The homology theory of the closed geodesic problem},
        date={1976},
        ISSN={0022-040X},
     journal={J. Differential Geometry},
      volume={11},
      number={4},
       pages={633\ndash 644},
  url={http://projecteuclid.org.acces.bibliotheque-diderot.fr/euclid.jdg/1214433729},
      review={\MR{455028}},
}

\end{biblist}
\end{bibdiv}

\vspace{20pt}

\end{document}